\documentclass[graybox]{svmult}

\usepackage{helvet}         
\usepackage{courier}        
\usepackage{type1cm}        
\usepackage{makeidx}         
\usepackage{graphicx}        
\usepackage{multicol}        
\usepackage[bottom]{footmisc}
\usepackage{bm}
\usepackage{mathtools}

\makeindex             

\usepackage{latexsym}
\usepackage{dsfont}
\usepackage{bbm}
\usepackage{amssymb}
\usepackage{amsmath}

\newtheorem{Theorem}{Theorem}[section]
\newtheorem{Lemma}[Theorem]{Lemma}
\newtheorem{Cor}[Theorem]{Corollary}
\newtheorem{Prop}[Theorem]{Proposition}
\newtheorem{Rem}[Theorem]{Remark}

\usepackage[top=4.cm, bottom=4cm, outer=3.8cm, inner=3.8cm]{geometry}

\def\cF{\mathcal{F}}

\def\cS{\mathcal{S}}

\def\Erw{\mathbb{E}}

\def\N{\mathbb{N}}
  
\def\Prob{\mathbb{P}} 

\def\R{\mathbb{R}}

\def\bS{\textit{\bfseries S}}

\def\bY{\hspace{-2.5pt}\textit{\bfseries Y}}

\def\eps{\varepsilon}

\def\vth{\vartheta}

\def\1{\vec{1}}
\def\3{{\ss}}

\def\eqdist{\stackrel{d}{=}}

\def\RA{\Rightarrow}

\def\wh{\widehat}

\def\dual{{}^{\#}\hspace{-1.1pt}}

\def\bM{\textit{\bfseries M}}

\allowdisplaybreaks[4] 

\begin{document}

\title*{An Arcsine Law for Markov Random Walks}
\titlerunning{An Arcsine Law for Markov Random Walks}
\author{Gerold Alsmeyer and Fabian Buckmann}
\institute{Inst.~Math.~Stochastics, Department
of Mathematics and Computer Science, University of M\"unster,
Orl\'eans-Ring 10, D-48149 M\"unster, Germany.\at
\email{gerolda@math.uni-muenster.de, f\_buck01@uni-muenster.de}}

\maketitle

\abstract{The classic arcsine law for the number $N_{n}^{>}:=n^{-1}\sum_{k=1}^{n}\1_{\{S_{k}>0\}}$ of positive terms, as $n\to\infty$, in an ordinary random walk $(S_{n})_{n\ge 0}$ is extended to the case when this random walk is governed by a positive recurrent Markov chain $(M_{n})_{n\ge 0}$ on a countable state space $\cS$, that is, for a Markov random walk $(M_{n},S_{n})_{n\ge 0}$ with positive recurrent discrete driving chain. More precisely, it is shown that $n^{-1}N_{n}^{>}$ converges in distribution to a generalized arcsine law with parameter $\rho\in [0,1]$ (the classic arcsine law if $\rho=1/2$) iff the Spitzer condition
$$ \lim_{n\to\infty}\frac{1}{n}\sum_{k=1}^{n}\Prob_{i}(S_{n}>0)\ =\ \rho $$
holds true for some and then all $i\in\cS$, where $\Prob_{i}:=\Prob(\cdot|M_{0}=i)$ for $i\in\cS$. It is also proved, under an extra assumption on the driving chain if $0<\rho<1$, that this condition is equivalent to the stronger variant
$$ \lim_{n\to\infty}\Prob_{i}(S_{n}>0)\ =\ \rho. $$
For an ordinary random walk, this was shown by Doney \cite{Doney:95} for $0<\rho<1$ and by Bertoin and Doney \cite{BertoinDoney:97} for $\rho\in\{0,1\}$.}

\bigskip



\section{Introduction}
The purpose of this note is to provide an arcsine law for the average number of positive sums
$S_{k}=\sum_{l=1}^{n}X_{l}$ up to time $n$ as $n\to\infty$ when the increments $X_{1},X_{2},\ldots$ are \emph{modulated} or \emph{driven} by a positive recurrent Markov chain $\bM=(M_{n})_{n\ge 0}$ with countable state space $\cS$. More precisely, the $X_{n}$ are conditionally independent given $\bM$, and
$$ \Prob((X_{1},\ldots,X_{n})\in\cdot\,|M_{0}=i_{0},\ldots,M_{n}=i_{n})\ =\ K_{i_{0}i_{1}}\otimes\ldots\otimes K_{i_{n-1}i_{n}} $$
for all $n\ge 1$, $i_{0},\ldots,i_{n}\in\cS$ and some stochastic kernel $K$ from $\cS^{2}$ to $\R$. Then $(M_{n},S_{n})_{n\ge 0}$, and sometimes also its additive part $(S_{n})_{n\ge 0}$, is called a \emph{Markov random walk (MRW)} or \emph{Markov additive process} and $\bM$ its \emph{driving chain}. Let $P=(p_{ij})_{i,j\in\cS}$ denote the transition matrix of $\bM$ and $\pi=(\pi_{i})_{i\in\cS}$ its unique stationary distribution. For any $i\in\cS$, we put further $\Prob_{i}:=\Prob(\cdot|M_{0}=i)$, $\Prob_{\pi}:=\sum_{i\in\cS}\pi_{i}\,\Prob_{i}$ and denote by $(\tau_{n}(i))_{n\ge 1}$ the renewal sequence of successive return epochs of  $\bM$ to $i$.

\vspace{.1cm}
If there exists a measurable function $g:\cS\to\R$ such that $K_{ij}=\delta_{g(j)-g(i)}$, thus
$X_{n}=g(M_{n})-g(M_{n-1})$ a.s. for all $n\in\N$, then the MRW is called \emph{null-homologous}, a term coined by Lalley \cite{Lalley:86}, and it is called \emph{nontrivial} otherwise. Here ``a.s.'' means $\Prob_{i}$-a.s. for all $i\in\cS$.

\vspace{.1cm}
In \cite{AlsBuck:17c}, a wide range of fluctuation-theoretic results for $(M_{n},S_{n})_{n\ge 0}$ has been established by the natural approach of drawing on corresponding results for the embedded ordinary random walks $(S_{\tau_{n}(i)})_{n\ge 1}$, $i\in\cS$, in combination with a thorough analysis of the excursions of the $S_{n}$ between the successive visits of the driving chain to a state $i$. Due to the fundamental observation that essential fluctuation-theoretic properties are shared by all embedded random walks (solidarity), the particular choice of $i$ does not matter for this approach. Here we will show that, if the limit
\begin{equation}\label{eq:Spitzer condition MRW}
\rho\ :=\ \lim_{n\to\infty}\frac{1}{n}\sum_{k=1}^{n}\Prob_{i}(S_{k}>0),
\end{equation}
exists for some $i\in\cS$, then it does so and is the same for any $i\in\cS$ (so we may replace $\Prob_{i}$ with $\Prob_{\pi}$), further satisfies
\begin{equation}\label{eq:Spitzer condition embedded RW}
\rho\ =\ \lim_{n\to\infty}\frac{1}{n}\sum_{k=1}^{n}\Prob_{i}(S_{\tau_{k}(i)}>0)
\end{equation}
for all $i\in\cS$, and entails that an arcsine law holds for
$$N_{n}^{>}\ :=\ \sum_{k=1}^{n}\1_{\{S_{k}>0\}} \quad\left( \text{and}\quad N_{n}^{\leqslant}\ :=\ n-N_{n}^{>}\ =\ \sum_{k=1}^{n}\1_{\{S_{k}\le 0\}}\right). $$
The precise statement of the result is given as Theorem \ref{thm:MRW arcsine} below. Validity of \eqref{eq:Spitzer condition embedded RW} is known as \emph{Spitzer's condition}
for the ordinary zero-delayed random walk $(S_{\tau_{n}(i)})_{n\ge 0}$ under $\Prob_{i}$, where $\tau_{0}(i):=0$, and is in fact equivalent to
\begin{equation}\label{eq:Doney condition embedded RW}
\rho\ =\ \lim_{n\to\infty}\Prob_{i}(S_{\tau_{n}(i)}>0),
\end{equation}
as shown by Doney \cite{Doney:95} for $0<\rho<1$, and by Bertoin and Doney \cite{BertoinDoney:97} for $\rho\in\{0,1\}$. Theorem \ref{thm:MRW arcsine} establishes also, under a second moment condition on $\tau(i)$ if $0<\rho<1$, the corresponding equivalence of \eqref{eq:Spitzer condition MRW} with
\begin{equation}\label{eq:Doney condition MRW}
\rho\ =\ \lim_{n\to\infty}\Prob_{i}(S_{n}>0)
\end{equation}
for any $i\in\cS$.

\vspace{.1cm}
Let $(AS(\theta))_{\theta\in[0,1]}$ be the family of generalized arcsine laws, i.e. $AS(0):=\delta_0$, $AS(1):=\delta_1$ and $AS(\theta)$ for $\theta\in(0,1)$ equals the beta distribution with parameters $1-\theta$ and $\theta$ and density
$$ \frac{\sin(\pi\,\theta)}{\pi}\,\frac{1}{x^{1-\theta}\,(1-x)^\theta}\,\mathbf{1}_{(0,1)}(x). $$
For $\theta=\frac{1}{2}$, we get the classical arcsine law with distribution function
$$ AS(1/2)((-\infty,x])\ =\ \frac{2}{\pi} \,\arcsin(\sqrt{x}),\quad x\in[0,1]. $$

\vspace{.1cm}
The following arcsine law for nontrivial MRW generalizes the corresponding classical result for ordinary random walk due to Spitzer \cite[Theorem 7.1]{Spitzer:56}, which in turn extended earlier versions by L\'evy \cite[Corollaire 2, p. 303]{Levy:40} and Sparre Andersen \cite[Theorem 3]{Andersen:54}.

\begin{Theorem}[Arcsine law for MRWs]\label{thm:MRW arcsine}
Let $(M_{n},S_{n})_{n\ge 0}$ be a nontrivial MRW with positive recurrent driving chain and consider the following assertions for arbitrary $i\in\cS$ and $\rho\in [0,1]$:
\begin{description}[(c)]\itemsep2pt
\item[(a)] Under $\Prob_{i}$,
\begin{equation}\label{eq:MRW arcsine}
\frac{N_{n}^{>}}{n}\ \xrightarrow{d}\ AS(\rho)\quad\text{and}\quad \frac{ N_{n}^{\leqslant}}{n}\ \xrightarrow{d}\  AS(1-\rho)
\end{equation}
as $n\to\infty$, where $\xrightarrow{d}$ means convergence in distribution.
\item[(b)] $\lim_{n\to\infty}n^{-1}\sum_{k=1}^{n}\Prob_{i}(S_{\tau_{k}(i)}>0)$ exists and equals $\rho$.
\item[(c)] Spitzer condition: $\lim_{n\to\infty}n^{-1}\sum_{k=1}^{n}\Prob_{i}(S_{k}>0)$ exists and equals $\rho$.
\item[(d)] Strong Spitzer condition: $\lim_{n\to\infty}\Prob_{i}(S_{n}>0)$ exists and equals $\rho$.
\end{description}
Then (a)--(c) are equivalent assertions and equivalence with (d) also holds true provided that $\Erw_{i}\tau(i)^{2}<\infty$ in the case $0<\rho<1$. Moreover, these assertions either hold for all $i\in\cS$ with the same $\rho$ or none.
\end{Theorem}

\begin{Rem}\label{rem:Freedman paper}\rm
The previous result, more precisely its implication ``(c)$\RA$(a)'', was already shown by Freedman \cite{Freedman:63} for the special case when $X_{n}=g(M_{n})$ for some measurable function $g$ and thus $S_{n}$ forms an additive functional of the driving chain. Regarding $g$, he further assumed $\Erw_{\pi}|X_{1}|=\sum_{i}\pi_{i}|g(i)|<\infty$, a condition not needed here.
\end{Rem}

\begin{Rem}\label{rem:CLT}\rm
In analogy to ordinary random walks, the classical arcsine law, that is \eqref{eq:MRW arcsine} with $\rho=\frac{1}{2}$, is obtained if $(S_{n})_{n\ge 0}$ satisfies a central limit theorem without centering, viz.
\begin{equation}\label{eq:CLT S_n}
\wh{S}_{n}\ :=\ \frac{S_{n}}{n^{1/2}}\ \xrightarrow{d}\ \textit{Normal}\,(0,\theta^{2})
\end{equation}
for some $\theta>0$. Namely, we then have
$$ \lim_{n\to\infty}\Prob(S_{n}>0)\ =\ \lim_{n\to\infty}\Prob(\wh{S}_{n}>0)\ =\ \frac{1}{2} $$
whence the assertion follows from part (d) of the above theorem. Note that $S_{n}$ may be viewed as an additive functional of the positive Harris chain $(M_{n},X_{n})_{n\ge 0}$. For such functionals, sufficient conditions for the validity of the central limit theorem, which typically include $\Erw_{\pi}X_{1}=0$ and $\Erw_{\pi}X_{1}^{2}<\infty$, have been studied by many authors, see e.g. Gordin and Lif\v sic \cite{GordinLifsic:78}, Woodroofe \cite{Woodroofe:92}, Maxwell and Woodroofe \cite{MaxwellWood:00}, Derriennic and Lin \cite{DerriennicLin:01} and also the references given therein.

Let us further mention that, in view of Condition (b) and by similar reasoning as before, the classical arcsine law $\textit{AS}(1/2)$ is also obtained if one (and by solidarity then all) of the ordinary embedded random walks $(S_{\tau_{n}(i)})_{n\ge 0}$ satisfies the central limit theorem without centering, which is well-known to be true if $\Erw_{i}S_{\tau(i)}=\Erw_{\pi}X_{1}\,\Erw_{i}\tau(i)=0$ and $\Erw_{i}S_{\tau(i)}^{2}<\infty$, thus if $(S_{n})_{n\ge 0}$ has stationary drift zero and finite variance over cycles determined by returns of the driving chain to a state $i$.
\end{Rem}

\begin{Rem}\rm
Albeit almost trivial, we note that $n^{-1}N_{n}^{>}$ either converges in distribution to a generalized arcsine law $\textit{AS}(\rho)$ or not at all. Namely, convergence to some law $Q$, say, on $[0,1]$ entails (by dominated convergence) that Theorem \ref{thm:MRW arcsine}(c) holds with $\rho:=\int x\,Q(dx)$ and thus $Q=\textit{AS}(\rho)$ by Theorem \ref{thm:MRW arcsine}(a).
\end{Rem}

\begin{Rem}\label{rem:reflected MRW}\rm
Since $(S_{\tau_{n}(i)})_{n\ge 0}$ is an ordinary random walk under $\Prob_{i}$, validity of Assertion (b) entails $\lim_{n\to\infty}n^{-1}\sum_{k=1}^{n}\Prob_{i}(S_{\tau_{k}(i)}<0)=1-\rho$
and thus validity of Theorem \ref{thm:MRW arcsine} for $(M_{n},-S_{n})_{n\ge 0}$ as well (with $1-\rho$ instead of $\rho$). As a particular consequence, we infer that
\begin{equation}\label{eq:S_n=0 behavior}
\lim_{n\to\infty}\frac{1}{n}\sum_{k=1}^{n}\Prob_{i}(S_{k}=0)\ =\ 0
\end{equation}
for all $i\in\cS$.
\end{Rem}

\begin{Rem}\label{rem:Spitzer condition P_pi}\rm
Let us further point out that Theorem \ref{thm:MRW arcsine}(c) for all $i\in\cS$ is also equivalent to
\begin{equation}\label{eq:Spitzer condition P_pi}
\lim_{n\to\infty}\frac{1}{n}\sum_{k=1}^{n}\Prob_{\pi}(S_{k}>0)\ =\ \rho.
\end{equation}
While the necessity of \eqref{eq:Spitzer condition P_pi} is obvious, the sufficiency proof needs a little more care and is deferred to Remark \ref{rem:Spitzer condition P_pi proof}.
\end{Rem}

\begin{Rem}\label{rem:extra assumption}\rm
Regarding the validity of the strong Spitzer condition (d), we do not know whether the additional assumption $\Erw_{i}\tau(i)^{2}<\infty$ in the case $0<\rho<1$ is really necessary but will provide an explanation in support of this in Remark \ref{rem2:extra assumption} at the end of Subsection \ref{subsec:0<rho<1}. On the other hand, the assumption is not very restrictive and particularly valid if the driving chain is geometrically ergodic or, a fortiori, has finite state space.
\end{Rem}

\begin{Rem}\label{rem:other arcsine laws}\rm
In the case of an ordinary random walk $(S_{n})_{n\ge 0}$, two further arcsine laws, namely for 
$$ L_{n}\,:=\,\min\{0\le k\le n: S_{k}=\max_{0\le l\le n}S_{l}\}\quad\text{and}\quad
L_{n}'\,:=\,\max\{0\le k\le n:S_{k}=\min_{0\le l\le n}S_{l}\}, $$
are directly derived by establishing $(N_{n}^{>},S_{n})\eqdist (L_{n},S_{n})\eqdist (L_{n}',S_{n})$ for all $n\in\N_{0}$, where $\,\eqdist\,$ means equality in law. Since these distributional identities are no longer at hand in the Markov-modulated situation, arcsine laws for $L_{n}$ and $L_{n}'$, if valid at all, require new arguments that will not be discussed here.
\end{Rem}

It is natural to expect, and confirmed by the next corollary, that the assertions of Theorem \ref{thm:MRW arcsine}, if valid for $(M_{n},S_{n})_{n\ge 0}$, also hold for the dual MRW $(\dual{M}_{n},\dual{S}_{n})_{n\ge 0}$. Recall that, in the notation given above, the dual chain $(\dual{M}_{n})_{n\ge 0}$ has transition matrix $\dual{P}=(\pi_{j}p_{ji}/\pi_{i})_{i,j\in\cS}$, while the conditional law of $\dual{X}_{n}=\dual{S}_{n}-\dual{S}_{n-1}$ given $\dual{M}_{n-1}=i,\dual{M}_{n}=j$ equals $\dual{K}_{ij}=K_{ji}$ for all $i,j\in\cS$. Since the embedded random walks $(S_{\tau_{n}(i)})_{n\ge 0}$ and $(\dual{S}_{\dual{\tau}_{n}(i)})_{n\ge 0}$ have the same distribution under $\Prob_{i}$ (w.l.o.g. put $\dual{M}_{0}=M_{0}$), we see that Theorem \ref{thm:MRW arcsine}(b), if valid for $(M_{n},S_{n})_{n\ge 0}$, also holds for the dual MRW.
The announced corollary is now immediate.

\begin{Cor}
If, for some $\rho\in [0,1]$ and some/all $i\in\cS$, the MRW $(M_{n},S_{n})_{n\ge 0}$ satisfies Theorem \ref{thm:MRW arcsine}(a)--(c), or (a)--(d), then the same holds true for its dual $(\dual{M}_{n},\dual{S}_{n})_{n\ge 0}$.
\end{Cor}

The further organization is as follows. The equivalence of Theorem \ref{thm:MRW arcsine}(a)--(c) is established in the next section, while Section \ref{sec:Assertion (d)} deals with a proof of the strong Spitzer condition (d) if (a)--(c) hold. As a crucial ingredient, for the case $0<\rho<1$, we will there derive an extension of a Spitzer formula which may be of independent interest, see Proposition \ref{prop:Spitzer-type lemma}.

\section{Proof of Theorem \ref{thm:MRW arcsine}(a)--(c)}

The proof of Theorem \ref{thm:MRW arcsine}(a)--(c) (in fact, their equivalence) will be furnished by a number of auxiliary lemmata the first of which is cited from \cite[Lemma 9.2]{AlsBuck:17a} and particularly shows that any nontrivial ordinary random walk $(S_{n})_{n\ge 0}$ converges to infinity in probability, a fact used in various places below.

\begin{Lemma}\label{lem:MRW concentration}
Let $(M_{n},S_{n})_{n\ge 0}$ be a nontrivial MRW having positive recurrent driving chain $M=(M_{n})_{n\ge 0}$ with stationary distribution $\pi$. Then $|S_{n}|\xrightarrow{\Prob_{\pi}}\infty$, i.e.
\begin{equation}\label{eq:MRW concentration}
\lim_{n\to\infty}\Prob_{\pi}(|S_{n}|\le x)\ =\ 0
\end{equation}
for all $x>0$.
\end{Lemma}

For the subsequent extension of Theorem \ref{thm:MRW arcsine}(a), we put
$$ N_{n}^{>}(x)\ :=\ \sum_{k=1}^{n}\1_{\{S_{k}>x\}}\quad\text{and}\quad N_{n}^{\leqslant}(x)\ :=\ \sum_{k=1}^{n}\1_{\{S_{k}\le x\}} $$
for $n\in\N$ and $x\in\R$.

\begin{Lemma}\label{lem:arcsine extended}
Let $(M_{n},S_{n})_{n\ge 0}$ be a nontrivial MRW with positive recurrent driving chain such that Theorem \ref{thm:MRW arcsine}(a) holds for some $i\in\cS$ and $\rho\in [0,1]$. Then under $\Prob_{i}$, as $n\to\infty$,
\begin{equation}\label{eq:arcsine extended}
\frac{N_{n}^{>}(x)}{n}\ \xrightarrow{d}\ AS(\rho)\quad\text{and}\quad\frac{N_{n}^{\leqslant}(x)}{n}\ \xrightarrow{d}\ AS(1-\rho)
\end{equation}
for all $x\in\R$.
\end{Lemma}

\begin{proof}
Plainly, it is enough to prove the first assertion. Since $n^{-1}N_{n}^{>}\xrightarrow{d}AS(\rho)$, it suffices to note that \eqref{eq:MRW concentration} of Lemma \ref{lem:MRW concentration} implies
$$ \frac{1}{n}\sum_{k=1}^{n}\1_{\{|S_{k}|\le x\}}\ \xrightarrow{\Prob}\ 0 $$
for all $x\ge 0$ and that
\begin{align*}
\frac{N_{n}^{>}}{n}\ -\ \frac{1}{n}\sum_{k=1}^{n}\1_{\{|S_{k}|\le |x|\}}\ \le\ \frac{N_{n}^{>}(x)}{n}\ \le\ \frac{N_{n}^{>}}{n}\ +\ \frac{1}{n}\sum_{k=1}^{n}\1_{\{|S_{k}|\le |x|\}}
\end{align*}
for all $x\in\R$.\qed
\end{proof}

A generalization of the classical arcsine law for ordinary random walks is next.

\begin{Lemma}\label{lem: RW general arcsine}
Let $(X_{n},Z_{n})_{n\ge 1}$ be a sequence of i.i.d. bivariate random vectors such that $\Prob(X_{1}=0)<1$ and $\Erw Z_{1}=\mu\in (0,\infty)$. Define $S_{0}:=0$ and $S_{n}:=\sum_{k=1}^{n}X_{k}$ for $n\ge 1$. If $(S_{n})_{n\ge 0}$ satisfies Spitzer's condition, i.e.
\begin{equation}\label{RW Spitzer Bedingung}
\rho\ :=\ \lim_{n\to\infty} \frac{1}{n}\sum_{k=1}^{n}\Prob(S_{k}>0),
\end{equation}
exists, then
\begin{equation}\label{eq:RW Allg arcsine}
\frac{1}{\mu n}\sum_{k=1}^{n}Z_{k}\,\1_{\{S_{k-1}>0\}}\ \xrightarrow{d}\ AS(\rho)\quad\text{and}\quad\frac{1}{\mu n}\sum_{k=1}^{n}Z_{k}\,\1_{\{S_{k-1}\le 0\}}\ \xrightarrow{d}\ AS(1-\rho)
\end{equation}
as $n\to\infty$.
\end{Lemma}

\begin{proof}
Since $\frac{1}{n}\sum_{k=1}^{n}Z_{k}\,\1_{\{S_{k-1}\le 0\}}=\frac{1}{n}\sum_{k=1}^{n}Z_{k}-\frac{1}{n}\sum_{k=1}^{n}Z_{k}\,\1_{\{S_{k-1}>0\}}$, we see that the two assertions in \eqref{eq:RW Allg arcsine} are equivalent and thus need to prove only the first one. We have
$n^{-1}N_{n-1}^{>}(x)\xrightarrow{d}AS(\rho)$ by the classical arcsine law and
$$ \frac{1}{n}\sum_{k=1}^{n}Z_{k}\,\1_{\{S_{k-1}>0\}}\ =\ \mu\,\frac{N_{n-1}^{>}}{n}\ +\ \frac{1}{n}\sum_{k=1}^{n}(Z_{k}-\mu)\,\1_{\{S_{k-1}>0\}}. $$
Hence it suffices to prove
$$ \frac{1}{n}\sum_{k=1}^{n}(Z_{k}-\mu)\,\1_{\{S_{k-1}>0\}}\ \xrightarrow{\Prob}\ 0 $$
as $n\to\infty$. But this follows directly from \cite[Thm.~2.19]{Hall+Heyde:80} when observing that the sequence $(\sum_{k=1}^{n}(Z_{k}-\mu)\,\1_{\{S_{k-1}>0\}})_{n\ge 0}$ forms a zero-mean martingale and
$$ \Prob(|Z_{k}-\mu|>z,S_{k-1}>0|\cF_{k-1})\ =\ \Prob(|Z_{1}-\mu|>z)\,\1_{\{S_{k-1}>0\}}\ \le\ \Prob(|Z_{1}-\mu|>z)\quad\text{a.s.} $$
for all $k\in\N$ and $z\ge 0$, where $(\cF_{n})_{n\ge 0}$ denotes the canonical filtration associated with $(X_{n},Z_{n})_{n\ge 1}$
\qed \end{proof}

For $n\in\N$ and $i\in\cS$, we put
\begin{align*}
D_{n}^{i}\ &:=\ \max_{\tau_{n-1}(i)<k\le\tau_{n}(i)}(S_{k}-S_{\tau_{n-1}(i)})^{-},\\
H_{n}^{i}\ &:=\ \max_{\tau_{n-1}(i)<k\le\tau_{n}(i)}(S_{k}-S_{\tau_{n-1}(i)})^{+}\\
\shortintertext{and}
\chi_{n}(i)\ &:=\, \tau_{n}(i)-\tau_{n-1}(i),
\end{align*}
where $\tau_{0}(i):=0$. Obviously, the triplets $(D_{n}^{i},H_{n}^{i},\chi_{n}(i))$ for $n\in\N$ are i.i.d. under $\Prob_{i}$.

\begin{Lemma}\label{lem:another arcsine law}
Let $(M_{n},S_{n})_{n\ge 0}$ be a nontrivial MRW with positive recurrent driving chain such that Theorem \ref{thm:MRW arcsine}(b) holds true for some $i\in\cS$ and $\rho\in [0,1]$. Put
$$ L_{n}^{i,>}\ :=\ \frac{1}{n}\sum_{k=1}^{n}\chi_{k}(i)\,\1_{\{0<S_{\tau_{k-1}(i)}\le D_{k}^{i}\}}\quad\text{and}\quad L_{n}^{i,\leqslant}\ :=\ \frac{1}{n}\sum_{k=1}^{n}\chi_{k}(i)\,\1_{\{-H_{k}^{i}<S_{\tau_{k-1}(i)}\le 0\}} $$
for $n\in\N$. Then $L_{n}^{i}:=L_{n}^{i,>}+L_{n}^{i,\leqslant}$ satisfies
\begin{equation}\label{eq:neglibility}
\lim_{n\to\infty}\Erw_{i}L_{n}^{i}\ =\ 0,
\end{equation}
in particular $L_{n}^{i}\xrightarrow{\Prob_{i}}0$. Moreover,
\begin{align}
\frac{\pi_{i}}{n}\sum_{k=1}^{n}\chi_{k}(i)\,\1_{\{S_{\tau_{k-1}(i)}-D_{k}^{i}>0\}}\ \xrightarrow{d}\ AS(\rho)\label{eq1:another arcsine law}
\shortintertext{and}
\frac{\pi_{i}}{n}\sum_{k=1}^{n}\chi_{k}(i)\,\1_{\{S_{\tau_{k-1}(i)}+H_{k}^{i}>0\}}\ \xrightarrow{d}\ AS(\rho)\label{eq2:another arcsine law}
\end{align}
under $\Prob_{i}$, as $n\to\infty$.
\end{Lemma}

\begin{proof}
As noted before Lemma \ref{lem:MRW concentration}, $S_{\tau_{n}(i)}\xrightarrow{\Prob_{i}}\infty$ and therefore
$$ \lim_{n\to\infty}\frac{1}{n}\sum_{k=1}^{n}\Prob(-y<S_{\tau_{k-1}(i)}\le x)\ =\ 0 $$
for all $x,y>0$. With the help of the dominated convergence theorem, this implies
\begin{align*}
&\lim_{n\to\infty}\Erw_{i}L_{n}^{i}\ =\ \lim_{n\to\infty}\frac{1}{n}\,\Erw_{i}\left(\sum_{k=1}^{n}\chi_{k}(i)\,\1_{\{-H_{k}^{i}<S_{\tau_{k-1}(i)}\le D_{k}^{i}\}}\right)\\
&\quad=\ \lim_{n\to\infty}\sum_{l\ge 1}\int\frac{1}{n}\sum_{k=1}^{n}\Prob(-y<S_{\tau_{k-1}(i)}\le x)\ l\,\Prob_{i}(\tau(i)=l,D_{1}^{i}\in dx,H_{1}^{i}\in dy)\ =\ 0,
\end{align*}
i.e. \eqref{eq:neglibility}. Since $\pi_{i}=(\Erw_{i}\tau(i))^{-1}$, we have by Lemma \ref{lem: RW general arcsine}, when applied to the sequence $(S_{\tau_{n}(i)}-S_{\tau_{n-1}(i)},\chi_{n}(i))_{n\ge 1}$, that
\begin{equation}\label{eq:arcsine law W_n}
W_{n}\ :=\ \frac{\pi_{i}}{n}\sum_{k=1}^{n}\chi_{k}(i)\,\1_{\{S_{\tau_{k-1}(i)}>0\}}\ \xrightarrow{d}\ AS(\rho)
\end{equation}
under $\Prob_{i}$, as $n\to\infty$. Observing that
\begin{align*}
\frac{\pi_{i}}{n}\sum_{k=1}^{n}\chi_{k}(i)\,\1_{\{S_{\tau_{k-1}(i)}-D_{k}^{i}>0\}}\ &=\ W_{n}-\pi_{i}\,L_{n}^{i,>}
\shortintertext{and}
\frac{\pi_{i}}{n}\sum_{k=1}^{n}\chi_{k}(i)\,\1_{\{S_{\tau_{k-1}(i)}+H_{k}^{i}>0\}}\ &=\ W_{n}+\pi_{i}\,L_{n}^{i,\leqslant}.
\end{align*}
the assertions \eqref{eq1:another arcsine law} and \eqref{eq2:another arcsine law} follow when combining $L_{n}^{i}\xrightarrow{\Prob_{i}}0$ with \eqref{eq:arcsine law W_n}.\qed
\end{proof}

\begin{Lemma}\label{lem:stopped arcsine law}
Let $(M_{n},S_{n})_{n\ge 0}$ be a nontrivial MRW with positive recurrent driving chain such that Theorem \ref{thm:MRW arcsine}(b) holds true for some $i\in\cS$ and $\rho\in [0,1]$. Then
\begin{equation}\label{eq:stopped arcsine law}
\frac{N_{\tau_{n}(i)}^{>}}{\tau_{n}(i)}\ \xrightarrow{d}\ AS(\rho)
\end{equation}
under $\Prob_{i}$, as $n\to\infty$. Moreover, the same holds true when replacing $\tau_{n}(i)$ with $\tau_{\Lambda(n)}(i)$ or $\tau_{\Lambda(n)+1}(i)$, where $\Lambda(n):=\sup\{k\ge 0:\tau_{k}(i)\le n\}$ for $n\in\N$.
\end{Lemma}

\begin{proof}
We first point out that
\begin{align*}
\{S_{\tau_{n-1}(i)}-D_{n}^{i}>0\}\ \subset\ \{S_{\tau_{n-1}(i)+k}>0\}\ \subset\ \{S_{\tau_{n-1}(i)}+H_{n}^{i}>0\}
\end{align*}
for all $n\in\N$ and $k=1,\ldots,\chi_{n}(i)$, hence
\begin{align}\label{eq:crucial inequality}
\frac{1}{n}\sum_{k=1}^{n}\chi_{k}(i)\,\1_{\{S_{\tau_{k-1}(i)}-D_{n}^{i}>0\}}\ \le\ \frac{N_{\tau_{n}(i)}^{>}}{n}\ \le\ \frac{1}{n}\sum_{k=1}^{n}\chi_{k}(i)\,\1_{\{S_{\tau_{k-1}(i)}+H_{k}^{i}>0\}}
\end{align}
for all $n\in\N$. Now use \eqref{eq:neglibility} in Lemma \ref{lem:another arcsine law} to infer that the difference of the upper and lower bound converges to 0 in $\Prob_{i}$-probability. Moreover, these bounds have the same asymptotic law by \eqref{eq1:another arcsine law} and \eqref{eq2:another arcsine law}, giving $\frac{\pi_{i}}{n}N_{\tau_{n}(i)}^{>}\xrightarrow{d}AS(\rho)$ under $\Prob_{i}$. Since $n^{-1}\tau_{n}(i)\to\pi_{i}^{-1}$ $\Prob_{i}$-a.s. by the strong law of large numbers, Slutsky's theorem implies \eqref{eq:stopped arcsine law}.

\vspace{.1cm}
Replacing $\tau_{n}(i)$ with $\tau_{\Lambda(n)}(i)$ or $\tau_{\Lambda(n)+1}(i)$, the same result is obtained by an appeal to Anscombe's theorem \cite[p.~16]{Gut:09} because
\begin{align*}
&\frac{\tau_{\Lambda(n)}}{n}\ \xrightarrow{n\to\infty}\ 1\quad\Prob_{i}\text{-a.s. entails}\quad\frac{\Lambda(n)}{n}\ =\ \frac{\Lambda(n)}{\tau_{\Lambda(n)}}\cdot\frac{\tau_{\Lambda(n)}}{n}\ \xrightarrow{n\to\infty}\ \frac{1}{\Erw_{i}\tau(i)}\ =\ \pi_{i}\quad\Prob_{i}\text{-a.s.}
\shortintertext{and}
&\forall\,\eps,\eta>0:\ \exists\,\delta>0,\,n_{0}\in\N:\ \forall\,n\ge n_{0}:\ \Prob_{i}\left(\max_{m:|m-n|<n\delta}\left|\frac{N_{\tau_{m}(i)}^{>}}{\tau_{m}(i)}-\frac{N_{\tau_{n}(i)}^{>}}{\tau_{n}(i)}\right|>\eps\right)\ <\ \eta
\end{align*}
as one can readily check.\qed
\end{proof}

\begin{Lemma}\label{lem:(c) implies (b)}
Let $(M_{n},S_{n})_{n\ge 0}$ be a nontrivial MRW with positive recurrent driving chain such that Theorem \ref{thm:MRW arcsine}(c) holds true for some $i\in\cS$ and $\rho\in [0,1]$. Then
Theorem \ref{thm:MRW arcsine}(b) for the same $i$ and $\rho$ is also valid.
\end{Lemma}

\begin{proof}
Keeping the notation from the previous lemma, notice that $\Lambda(n)\le n$ and recall that $n^{-1}\Lambda(n)\to\pi_{i}$ $\Prob_{i}$-a.s.
As a consequence, $\Lambda'(n):=\Lambda(n)\vee n_{\eps}$ and $\Lambda''(n):=\Lambda(n)\wedge n^{\eps}$, where $n_{\eps}:=\lceil(1-\eps)\pi_{i}n\rceil$ and $n^{\eps}:=\lceil(1+\eps)\pi_{i}n\rceil$ for any fixed $\eps\in (0,1)$, satisfy $\Lambda'(n)-\Lambda(n)\to 0$ and $\Lambda(n)-\Lambda''(n)\to 0$ $\Prob_{i}$-a.s. Moreover, for any stopping time $\nu$ for the sequence $(\tau_{k}(i),S_{\tau_{k}(i)})_{k\ge 0}$, the identity
\begin{align}
\begin{split}\label{eq:aux4}
&\Erw_{i}\left(\sum_{k=1}^{\nu}\chi_{k}(i)\1_{\{S_{\tau_{k-1}(i)}>0\}}\right)\ =\ \sum_{k\ge 1}\Erw_{i}\chi_{k}(i)\1_{\{S_{\tau_{k-1}(i)}>0,\nu>k-1\}}\\
&\hspace{2cm}=\ \Erw_{i}\tau(i)\sum_{k\ge 1}\Prob_{i}(S_{\tau_{k-1}(i)}>0,\nu>k-1)\\
&\hspace{2cm}=\ \frac{1}{\pi_{i}}\,\Erw_{i}\left(\sum_{k=1}^{\nu}\1_{\{S_{\tau_{k-1}(i)}>0\}}\right)
\end{split}
\end{align}
holds true and will be utilized hereafter for $\nu=\Lambda(n)+1=\inf\{k:\tau_{k}(i)>n\}$. Also noting that
\begin{align}
&\hspace{4cm}\lim_{n\to\infty}\frac{1}{n}\,\Erw_{i}\chi_{\Lambda(n)+1}(i)\ =\ 0,\label{eq:aux1}\\
&\lim_{n\to\infty}\frac{1}{n}\,\Erw_{i}\left(\sum_{k=1}^{\Lambda(n)+1}\chi_{k}(i)\,\1_{\{-H_{k}^{i}<S_{\tau_{k-1}(i)}\le D_{k}^{i}\}}\right)\ \le\ \lim_{n\to\infty}\frac{n+1}{n}\,\Erw_{i}L_{n+1}^{i}\ =\ 0\label{eq:aux2},
\shortintertext{and}
&\hspace{1.5cm}\lim_{n\to\infty}\frac{1}{n}\,\Erw_{i}(\Lambda'(n)-\Lambda(n))\ =\ \lim_{n\to\infty}\frac{1}{n}\,\Erw_{i}(\Lambda''(n)-\Lambda(n))\ =\ 0,\label{eq:aux3}
\end{align}
we now infer
\begin{align*}
\rho\ &=\ \lim_{n\to\infty}\frac{\Erw_{i}N_{n}^{>}}{n}\ \stackrel{\eqref{eq:aux1}}{=}\ \lim_{n\to\infty}\frac{1}{n}\,\Erw_{i}\left(\sum_{k=1}^{\tau_{\Lambda(n)+1}(i)}\1_{\{S_{k}>0\}}\right)\\
&\le\ \liminf_{n\to\infty}\frac{1}{n}\,\Erw_{i}\left(\sum_{k=1}^{\Lambda(n)+1}\chi_{k}(i)\1_{\{S_{\tau_{k-1}(i)}+H_{k}^{i}>0\}}\right)\\
&\stackrel{\eqref{eq:aux2}}{=}\ \liminf_{n\to\infty}\frac{1}{n}\,\Erw_{i}\left(\sum_{k=1}^{\Lambda(n)+1}\chi_{k}(i)\1_{\{S_{\tau_{k-1}(i)}>0\}}\right)\\
&\stackrel{\eqref{eq:aux4}}{=}\ \liminf_{n\to\infty}\frac{1}{\pi_{i}n}\,\Erw_{i}\left(\sum_{k=1}^{\Lambda(n)+1}\1_{\{S_{\tau_{k-1}(i)}>0\}}\right)\\
&\stackrel{\eqref{eq:aux3}}{=}\ \liminf_{n\to\infty}\frac{1}{\pi_{i}n}\,\Erw_{i}\left(\sum_{k=1}^{\Lambda''(n)+1}\1_{\{S_{\tau_{k-1}(i)}>0\}}\right)\\
&\le\ (1+\eps)\liminf_{n\to\infty}\frac{1}{n^{\eps}}\sum_{k=1}^{n^{\eps}}\Prob_{i}(S_{\tau_{k-1}(i)}>0)
\shortintertext{and, conversely,}
\rho\ &\ge\ \limsup_{n\to\infty}\frac{1}{n}\,\Erw_{i}\left(\sum_{k=1}^{\Lambda(n)+1}\chi_{k}(i)\1_{\{S_{\tau_{k-1}(i)}-D_{k}^{i}>0\}}\right)\\
&\stackrel{\eqref{eq:aux2}}{=}\ \limsup_{n\to\infty}\frac{1}{n}\,\Erw_{i}\left(\sum_{k=1}^{\Lambda(n)+1}\chi_{k}(i)\1_{\{S_{\tau_{k-1}(i)}>0\}}\right)\\
&\stackrel{\eqref{eq:aux4}}{=}\ \limsup_{n\to\infty}\frac{1}{\pi_{i}n}\,\Erw_{i}\left(\sum_{k=1}^{\Lambda(n)+1}\1_{\{S_{\tau_{k-1}(i)}>0\}}\right)\\
&\stackrel{\eqref{eq:aux3}}{=}\ \limsup_{n\to\infty}\frac{1}{\pi_{i}n}\,\Erw_{i}\left(\sum_{k=1}^{\Lambda'(n)+1}\1_{\{S_{\tau_{k-1}(i)}>0\}}\right)\\
&\ge\ (1-\eps)\limsup_{n\to\infty}\frac{1}{n_{\eps}}\sum_{k=1}^{n_{\eps}}\Prob_{i}(S_{\tau_{k-1}(i)}>0).
\end{align*}
Since $\eps\in (0,1)$ was arbitrarily chosen and $\{n_{\eps}:n\in\N\}=\{n^{\eps}:n\in\N\}=\N$ for all sufficiently small $\eps$, we infer validity of Theorem \ref{thm:MRW arcsine}(b).\qed
\end{proof}

\begin{proof}[of Theorem \ref{thm:MRW arcsine}(a)--(c)]
Fix any $i\in\cS$. Then (a) implies (c) by taking expectations, and (c) implies (b) by Lemma \ref{lem:(c) implies (b)}. To see that (b) implies (a), note first that Lemma \ref{lem:stopped arcsine law} provides us with
$$ \frac{N_{\tau_{\Lambda(n)}(i)}^{>}}{\tau_{\Lambda(n)}(i)}\ \xrightarrow{d}\ AS(\rho)\quad\text{and}\quad\frac{N_{\tau_{\Lambda(n)+1}(i)}^{>}}{\tau_{\Lambda(n)+1}(i)}\ \xrightarrow{d}\ AS(\rho) $$
as $n\to\infty$. The assertion now follows because $N_{\tau_{\Lambda(n)}(i)}^{>}\le N_{n}^{>}\le N_{\tau_{\Lambda(n)+1}(i)}^{>}$ and $\tau_{\Lambda(n)}(i)\le n\le\tau_{\Lambda(n)+1}(i)$, thus
$$ \frac{N_{\tau_{\Lambda(n)}(i)}^{>}}{\tau_{\Lambda(n)+1}(i)}\ \le\ \frac{N_{n}^{>}}{n}\ \le\ \frac{N_{\tau_{\Lambda(n)+1}(i)}^{>}}{\tau_{\Lambda(n)}(i)}, $$
and $\tau_{\Lambda(n)+1}(i)/\tau_{\Lambda(n)}(i)\to 1$ $\Prob_{i}$-a.s.

\vspace{.1cm}
In order to show that (a)--(c) hold under $\Prob_{j}$ for any $j\ne i$ as well, pick any $\eps>0$ and an integer sequence $(k_{n})_{n\ge 1}$ such that $k_{n}\to \infty$ and $n^{-1}k_{n}\to 0$. Fix $j\ne i$ and choose $x>0$ so large that $\Prob_{j}(|S_{\tau(i)}|>x)<\eps$. Then
\begin{align*}
\frac{1}{n}\sum_{k=1}^{n}\,&\Prob_{j}(S_{k}>0)\ \le\ \frac{1}{n}\,\Erw_{j}\left(\1_{\{\tau(i)\le k_{n}\}}\sum_{k=\tau(i)}^{n}\1_{\{S_{k}>0\}}\right)\,+\,\frac{k_{n}}{n}\,+\,\Prob_{j}(\tau(i)>k_{n})\\
&\le\ \frac{1}{n}\,\Erw_{j}\left(\1_{\{\tau(i)\le k_{n},|S_{\tau(i)}|\le x\}}\sum_{k=\tau(i)}^{n}\1_{\{S_{k}>0\}}\right)\,+\,\frac{k_{n}}{n}\,+\,\Prob_{j}(\tau(i)>k_{n})\,+\,\eps\\
&\le\ \frac{1}{n}\sum_{k=1}^{n}\Prob_{i}(S_{k}>-x)\,+\,\frac{k_{n}}{n}\,+\,\Prob_{j}(\tau(i)>k_{n})\,+\,\eps.
\end{align*}
Use Lemma \ref{lem:arcsine extended} to see that
$$ \lim_{n\to\infty}\frac{1}{n}\sum_{k=1}^{n}\Prob_{i}(S_{k}>y)\ =\ \lim_{n\to\infty}\frac{\Erw_{i}N_{n}^{>}(y)}{n}\ =\ \rho $$
for all $y\in\R$. Therefore,
$$ \limsup_{n\to\infty}\frac{1}{n}\sum_{k=1}^{n}\Prob_{j}(S_{k}>0)\ \le\ \rho+\eps. $$
By a similar argument, we find
\begin{align*}
\frac{1}{n}\sum_{k=1}^{n}\Prob_{j}(S_{k}>0)\ &\ge\ \frac{1}{n}\,\Erw_{j}\left(\1_{\{\tau(i)\le k_{n},|S_{\tau(i)}|\le x\}}\sum_{k=\tau(i)}^{n}\1_{\{S_{k}>0\}}\right)\\
&\ge\ \Prob_{j}(\tau(i)\le k_{n},|S_{\tau(i)}|\le x)\left(\frac{1}{n}\sum_{k=1}^{n-k_{n}}\Prob_{i}(S_{k}>x)\right)\\
&\ge\ \big(1-\Prob_{j}(\tau(i)>k_{n})-\eps\big)\left(\frac{1}{n}\sum_{k=1}^{n}\Prob_{i}(S_{k}>x)\,-\,\frac{k_{n}}{n}\right)
\end{align*}
and thereby
$$ \liminf_{n\to\infty}\frac{1}{n}\sum_{k=1}^{n}\Prob_{j}(S_{k}>0)\ \ge\ (1-\eps)\rho. $$
Since $\eps>0$ was arbitrary, we conclude validity of Theorem \ref{thm:MRW arcsine}(c) for $j\ne i$ and thus also of (a) and (b) by the first part of the proof.\qed
\end{proof}

\begin{Rem}\label{rem:Spitzer condition P_pi proof}\rm
By adapting the previously given argument, it is now easily proved that \eqref{eq:Spitzer condition P_pi} implies Theorem \ref{thm:MRW arcsine}(c). First note that, by Lemma \ref{lem:MRW concentration}, we have
$$ \frac{1}{n}\sum_{k=1}^{n}\Prob_{\pi}(S_{k}>x)\ =\ \rho $$
for all $x\in\R$. Fix $i\in\cS$ and pick an arbitrary $\eps\in (0,1)$. Choose $(k_{n})_{n\ge 1}$ as above and $x$ such that $\Prob_{\pi}(|S_{\tau(i)}|\le x)<\eps$. Then
\begin{align*}
&\frac{1}{n}\sum_{k=1}^{n}\,\Prob_{\pi}(S_{k}>x)\ \le\ \frac{1}{n}\sum_{k=1}^{n}\Prob_{i}(S_{k}>0)\,+\,\frac{k_{n}}{n}\,+\,\Prob_{j}(\tau(i)>k_{n})\,+\,\eps
\shortintertext{and}
&\frac{1}{n}\sum_{k=1}^{n}\Prob_{\pi}(S_{k}>-x)\ \ge\ (1-\Prob_{\pi}(\tau(i)>k_{n})-\eps)\left(\frac{1}{n}\sum_{k=1}^{n}\Prob_{i}(S_{k}>0)\,-\,\frac{k_{n}}{n}\right)
\end{align*}
and from this one easily concludes Theorem \ref{thm:MRW arcsine}(c) for the chosen $i\in\cS$.
\end{Rem}

\section{The strong Spitzer condition: Proof of Theorem \ref{thm:MRW arcsine}(d)}\label{sec:Assertion (d)}

\subsection{The case $0<\rho<1$}\label{subsec:0<rho<1}

For an ordinary random walk $(S_{n})_{n\ge 0}$, Doney's \cite{Doney:95} proof of the equivalence of the Spitzer condition and its strong version is based on the Spitzer-type formula (see \cite[Eq.~(7.7) on p.~414]{Feller:71})
\begin{equation}\label{eq:Spitzer formula}
\Prob(S_{n}>0)\ =\ \sum_{k\ge 1}\frac{n}{k}\,\Prob(\sigma(k)=n)
\end{equation}
where $\sigma(k)$ denotes the (possibly defective) $k$-th strictly ascending ladder epoch of $(S_{n})_{n\ge 0}$. The subsequent proposition provides a substitute for this formula in the Markov-modulated situation which again uses Spitzer's combinatorial argument but for the i.i.d. blocks defined by the successive returns of the driving chain to an arbitrarily fixed state.

\begin{Prop}\label{prop:Spitzer-type lemma}
Let $(M_{n},S_{n})_{n\ge 0}$ be a MRW with positive recurrent driving chain on $\cS$. For any fixed $i\in\cS$, let $(\sigma(n))_{n\ge 1}$ be the (possibly defective) sequence of strictly ascending ladder epochs of the embedded random walk $(S_{\tau_{n}(i)})_{n\ge 0}$ under $\Prob_{i}$. Then
\begin{equation}\label{eq:Spitzer-type lemma}
\Prob_{i}(M_{n}=i,S_{n}>0)\ =\ \sum_{k\ge 1}\frac{n}{k}\,\Erw_{i}\left(\frac{\sigma(k)}{\tau_{\sigma(k)}(i)}\1_{\{\tau_{\sigma(k)}(i)=n\}}\right).
\end{equation}
for all $n\in\N$.
\end{Prop}

Notice that \eqref{eq:Spitzer-type lemma} reduces to \eqref{eq:Spitzer formula} as it must if $(M_{n})_{n\ge 0}$ is a single-state Markov chain and thus $(S_{n})_{n\ge 0}$ an ordinary random walk.

\begin{proof}
For fixed $m,n\in\N$, consider the event $A_{m,n}:=\{\tau_{m}(i)=n,S_{n}>0\}$ and note that $\1_{A_{m,n}}=\1_{B_{m,n}}(Y_{1},\ldots,Y_{m})$, where
\begin{align*}
&B_{m,n}\ :=\ \left\{(i_{j},x_{j})_{1\le j\le n}\in (\cS\times\R)^{n}:i_{n}=i,\sum_{r=1}^{n}\1_{\{i\}}(i_{r})=m\text{ and }\sum_{r=1}^{n}x_{r}>0\right\}
\shortintertext{and}
&\hspace{3cm}Y_{l}\ :=\ \left(M_{j},X_{j}\right)_{\tau_{l-1}(i)+1\le j\le\tau_{l}(i)},\quad l\in\N.
\end{align*}
Put $\bY^{\,1}\!:=(Y_{1},\ldots,Y_{m}),\bY^{\,2}\!:=(Y_{m},Y_{1},\ldots,Y_{m-1}),\ldots,\bY^{\,m}:=(Y_{2},\ldots,Y_{m},Y_{1})$, which are the cyclic rearrangements of the i.i.d. block vectors $Y_{1},\ldots,Y_{m}$ and thus identically distributed (under $\Prob_{i}$). Denote by 
$$ \bS^{\,l}\ =\ \big(S_{1}^{l},\ldots,S_{n}^{l}\big),\quad l=1,\ldots,m $$
the resulting vectors of partial sums after the rearrangements, thus $\bS^{\,1}:=(S_{1},\ldots,S_{n})$. Notice that  $\1_{B_{m,n}}(\bY^{\,1})=\ldots=\1_{B_{m,n}}(\bY^{\,m})$ and $S_{n}^{l}=S_{n}$ for each $l=1,\ldots,m$.

\vspace{.1cm}
Now fix any $k\in\N$ and suppose that $k$ is the number of strict record values among those $S_{j}$ in $\bS^{\,1}$ with $M_{j}=i$, in other words, the number of strictly ascending ladder heights in $\{S_{\tau_{1}(i)},\ldots,S_{\tau_{m-1}(i)},S_{\tau_{m}(i)}\}$, i.e. $\sigma(k)\le m<\sigma(k+1)$. We can write this event as $(\bS^{\,1})^{-1}(B_{k,m,n})$ for some $B_{k,m,n}\subset B_{m,n}$. The crucial fact to be used hereafter is that the number $k$ does not vary for the vectors $\bS^{\,l}$, thus $(\bS^{\,1})^{-1}(B_{k,m,n})=\ldots=(\bS^{\,m})^{-1}(B_{k,m,n})$, and that $k$ is also the number of these vectors for which the terminal value $S_{n}^{l}$ is a record. This follows by a simple combinatorial argument (see \cite[Lemma 1 on p.~412]{Feller:71}). Defining $I_{k}^{l}:=1$ if $\bY^{\,l}\in B_{k,m,n}$ and $S_{n}^{l}$ is a record value, and $I_{k}^{l}:=0$ otherwise, it follows that $I_{k}^{1}+\ldots+I_{k}^{m}$ takes only the two values $0$ and $k$. Since $I_{k}^{1},\ldots,I_{k}^{m}$ are also identically distributed under $\Prob_{i}$ with
$$ \Erw_{i}I_{k}^{1}\ =\ \Prob_{i}(A_{m,n}\cap\{\sigma(k)=m\})\ =\ \Prob_{i}(\sigma(k)=m,\tau_{\sigma(k)}(i)=n), $$
we arrive at
\begin{align*}
\Prob_{i}(\sigma(k)=m,\tau_{\sigma(k)}(i)=n)\ &=\ \frac{1}{m}\,\Erw_{i}(I_{k}^{1}+\ldots+I_{k}^{m})\ =\ \frac{k}{m}\,\Prob_{i}(I_{k}^{1}+\ldots+I_{k}^{m}=k).
\end{align*}
On the other hand, the events $\{I_{k}^{1}+\ldots+I_{k}^{m}=k\}$ for $k\in\N$ are pairwise disjoint and their union is $A_{m,n}$, hence
\begin{align*}
\Prob_{i}(A_{m,n})\ &=\ \sum_{k\ge 1}\frac{m}{k}\,\Prob_{i}(\sigma(k)=m,\tau_{\sigma(k)}(i)=n)\\
&=\ \sum_{k\ge 1}\Erw_{i}\left(\frac{\sigma(k)}{k}\1_{\{\sigma(k)=m,\tau_{\sigma(k)}(i)=n\}}\right)\\
&=\ \sum_{k\ge 1}\frac{n}{k}\,\Erw_{i}\left(\frac{\sigma(k)}{\tau_{\sigma(k)}(i)}\1_{\{\sigma(k)=m,\tau_{\sigma(k)}(i)=n\}}\right).
\end{align*}
Now the assertion \eqref{eq:Spitzer-type lemma} follows upon summing both sides over $m\in\N$ and using that the left-hand side then equals $\Prob_{i}(M_{n}=i,S_{n}>0)$.\qed
\end{proof}

Formula \eqref{eq:Spitzer-type lemma} forms the key ingredient to the following lemma which in turn furnishes our proof of Theorem \ref{thm:MRW arcsine}(d) in the case $0<\rho<1$.

\begin{Lemma}\label{lem:S_n>0,M_n=i}
Let $(M_{n},S_{n})_{n\ge 0}$ be a MRW with positive recurrent driving chain satisfying $\Erw_{i}\tau(i)^{2}<\infty$ and Theorem \ref{thm:MRW arcsine}(a)--(c) for some $\rho\in (0,1)$ and all $i\in\cS$. Then
\begin{equation}\label{lem:S_n>0,M_n=i}
\lim_{n\to\infty}\Prob_{i}(M_{nd}=i,S_{nd}>0)\ =\ d\rho\pi_{i}
\end{equation}
all $i\in\cS$, where $d\in\N$ denotes the period of $(M_{n})_{n\ge 0}$.
\end{Lemma}

\begin{proof}
We may restrict ourselves to the case when $(M_{n})_{n\ge 0}$ is aperiodic, thus $d=1$.
Fix any $i\in\cS$ and let $\Prob_{i}$ be the underlying probability measure. If Theorem \ref{thm:MRW arcsine}(b) holds, then $\sigma(1)$ lies in the domain of attraction of $\mathsf{S}(\rho)$, the one-sided stable law with index $\rho$ (see e.g.\ \cite[Thm. 8.9.12]{BingGolTeug:89}), and since $\tau_{\sigma(n)}/\sigma(n)\to\pi_{i}^{-1}$ $\Prob_{i}$-a.s., the same holds true for $\tau_{\sigma(1)}$. In fact, we can choose a continuous increasing function $\vth:[0,\infty)\to [0,\infty)$ which has inverse $\vth^{-1}$ and is regularly varying with index $1/\rho$ at $\infty$ such that $\tau_{\sigma(n)}/\vth(n)$ converges in distribution to $\mathsf{S}(\rho)$. Let $f$ denote its density. By making use of the local limit theorem of Gnedenko (see \cite[Thm.~4.2.1]{Ibragimov+Linnik:71}), Doney \cite{Doney:95} showed that for all $\delta>0$
\begin{align*}
&\sum_{\vth^{-1}(\delta n)\le k\le\vth^{-1}(n/\delta)}\frac{n}{k\vth(k)}\,\Prob_{i}(\tau_{\sigma(k)}(i)=n)\\
&\hspace{1cm}=\ \sum_{\vth^{-1}(\delta n)\le k\le\vth^{-1}(n/\delta)}\frac{n}{k\vth(k)}\,f\left(\frac{n}{\vth(k)}\right)\ +\ o(1)\\
&\hspace{1cm}=\ \rho\int_{\delta}^{1/\delta}f(x)\ dx\ +\ o(1)
\end{align*}
as $n\to\infty$. Using this, we infer that, for any $\delta,\eps>0$,
\begin{align*}
\Prob_{i}&(M_{n}=i,S_{n}>0)\\
&\ge\ \sum_{\vth^{-1}(\delta n)\le k\le\vth^{-1}(n/\delta)}\frac{n(\pi_{i}-\eps)}{k}\,\Prob_{i}\left(\frac{\sigma(k)}{\tau_{\sigma(k)}(i)}\ge\pi_{i}-\eps,\,\tau_{\sigma(k)}(i)=n\right)\\
&\ge\ (\pi_{i}-\eps)\left[\sum_{\vth^{-1}(\delta n)\le k\le\vth^{-1}(n/\delta)}\frac{n}{k}\,\Prob_{i}\left(\tau_{\sigma(k)}(i)=n\right)\,-\,R_{n}(\delta)\right]\\
&=\ (\pi_{i}-\eps)\left[\rho\int_{\delta}^{1/\delta}f(x)\ dx\,-\,R_{n}(\delta)\right]\ +\ o(1)
\end{align*}
as $n\to\infty$, where
\begin{align*}
R_{n}(\delta)\ &:=\ \sum_{\vth^{-1}(\delta n)\le k\le\vth^{-1}(n/\delta)}\frac{n}{k}\,\Prob_{i}\left(\frac{\sigma(k)}{\tau_{\sigma(k)}(i)}<\pi_{i}-\eps,\,\tau_{\sigma(k)}(i)=n\right)\\
&\le\ \sum_{\vth^{-1}(\delta n)\le k\le\vth^{-1}(n/\delta)}\frac{n}{k}\,\Prob_{i}\left(\eps<\frac{\sigma(k)}{\tau_{\sigma(k)}(i)}<\pi_{i}-\eps,\,\tau_{\sigma(k)}(i)=n\right)\\
&\qquad+\ \sum_{\vth^{-1}(\delta n)\le k\le\vth^{-1}(n/\delta)}\frac{n}{k}\,\Prob_{i}\left(\sigma(k)<\eps n,\,\tau_{\sigma(k)}(i)=n\right).
\end{align*}
But for $\vth^{-1}(\delta n)\le k\le\vth^{-1}(n/\delta)$, we further find that
\begin{align*}
&\Prob_{i}\left(\eps<\frac{\sigma(k)}{\tau_{\sigma(k)}(i)}<\pi_{i}-\eps,\,\tau_{\sigma(k)}(i)=n\right)\\
&\hspace{1.5cm}\le\ \Prob_{i}\left(\frac{\tau_{\sigma(k)}(i)-\pi_{i}^{-1}\sigma(k)}{\sigma(k)}>\frac{1}{\pi_{i}-\eps}-\frac{1}{\pi_{i}},\sigma(k)>\eps n,\tau_{\sigma(k)}(i)=n\right)\\
&\hspace{1.5cm}\le\ \Prob_{i}\left(\sup_{m\ge\eps n}\frac{\tau_{m}(i)-m\pi_{i}^{-1}}{m}>\frac{\eps}{\pi_{i}(\pi_{i}-\eps)}\right)
\shortintertext{and}
&\hspace{1cm}\Prob_{i}\left(\sigma(k)<\eps n,\,\tau_{\sigma(k)}(i)=n\right)\ \le\ \Prob_{i}\left(\frac{\tau_{\eps n}(i)-\pi_{i}^{-1}\eps n}{\eps n}>\frac{1}{\eps}-\frac{1}{\pi_{i}}\right),
\end{align*}
giving
\begin{align*}
R_{n}(\delta)\ &\le\ \frac{n\big(\vth^{-1}(n/\delta)-\vth^{-1}(\delta n)\big)}{\vth^{-1}(\delta n)}\,\Prob_{i}\left(\sup_{m\ge\eps n}\frac{\tau_{m}(i)-m\pi_{i}^{-1}}{m}>\eps\right)\\
&\le\ Cn\,\Prob_{i}\left(\sup_{m\ge\eps n}\frac{\tau_{m}(i)-m\pi_{i}^{-1}}{m}>\frac{\eps}{\pi_{i}}\right)
\end{align*}
for some $C>0$ and any $\eps>0$ sufficiently small (and with the convention that $\tau_{x}(i):=\tau_{\lceil x\rceil}(i)$). But, for any $\eps>0$, 
$$ a_{n}\ :=\ n\,\Prob_{i}\left(\sup_{m\ge n}\frac{\tau_{m}(i)-m\pi_{i}^{-1}}{m}>\eps\right)\ \xrightarrow{n\to\infty}\ 0 $$
because $\Erw_{i}\tau(i)^{2}<\infty$ ensures $\sum_{n\ge 1}n^{-1}a_{n}<\infty$, see Chow and Lai \cite[Eq. (3.10) with $\alpha=1$ and $p=2$]{Chow+Lai:75}, and $a_{n}/n$ is nonincreasing. By combining the previous estimates and noting that $\int_{\delta}^{1/\delta}f(x)\,dx\to 1$ as $\delta\to 0$, we conclude
\begin{equation}\label{eq:liminf estimate}
\liminf_{n\to\infty}\Prob_{i}(M_{n}=i,S_{n}>0)\ \ge\ \pi_{i}\rho.
\end{equation}
In view of Remark \ref{rem:reflected MRW}, we can repeat the argument for $(M_{n},-S_{n})_{n\ge 0}$ to obtain
\begin{equation*}
\liminf_{n\to\infty}\Prob_{i}(M_{n}=i,S_{n}<0)\ \ge\ \pi_{i}(1-\rho)
\end{equation*}
or, equivalently,
\begin{equation}\label{eq:limsup estimate}
\limsup_{n\to\infty}\Prob_{i}(M_{n}=i,S_{n}\ge 0)\ \le\ \pi_{i}\rho.
\end{equation}
Finally, \eqref{lem:S_n>0,M_n=i} follows by a combination of \eqref{eq:liminf estimate} and \eqref{eq:limsup estimate}.\qed
\end{proof}

\begin{proof}[of Theorem \ref{thm:MRW arcsine}(d)]
Assertion (d) is now easily derived as follows. Fix any $i\in\cS$ and suppose first that the driving chain is aperiodic $(d=1)$. Then we obtain with the help of \eqref{lem:S_n>0,M_n=i} and Lemma \ref{lem:MRW concentration} that
\begin{align*}
\Prob_{i}(M_{n}=j,S_{n}>0)\ &=\ \sum_{k=1}^{n}\int\Prob_{j}(M_{n-k}=j,S_{n-k}>-x)\ \Prob_{i}(\tau(j)=k,S_{k}\in dx)\\
&=\ \sum_{k=1}^{\lfloor n/2\rfloor}\int\Prob_{j}(M_{n-k}=j,S_{n-k}>0)\ \Prob_{i}(\tau(j)=k,S_{k}\in dx)\ +\ o(1)\\
&=\ \pi_{j}\rho\ +\ o(1)
\end{align*}
as $n\to\infty$ and thereupon $\lim_{n\to\infty}\Prob_{i}(S_{n}>0)=\rho$ by summation over $j$.

\vspace{.1cm}
If $(M_{n})_{n\ge 0}$ has period $d\ge 2$, then let $\cS_{r}$, $r=0,\ldots,d-1$, denote the cyclic class of states that can be reached from $i$ at times $nd+r$ for $n\in\N_{0}$. For $j\in\cS_{r}$, it then follows in a similar manner as before that
\begin{align*}
&\Prob_{i}(M_{nd+r}=j,S_{nd+r}>0)\\
&=\ \sum_{k=0}^{\lfloor n/2\rfloor}\int\Prob_{j}(M_{(n-k)d}=j,S_{(n-k)d}>0)\ \Prob_{i}(\tau(j)=kd+r,S_{kd+r}\in dx)\ +\ o(1)\\
&=\ d\pi_{j}\rho\ +\ o(1)
\end{align*}
as $n\to\infty$ and thereupon, using $\pi(\cS_{r})=d^{-1}$,
$$ \Prob_{i}(S_{nd+r}>0)\ =\ \sum_{j\in\cS_{r}}\Prob_{i}(M_{nd+r}=j,S_{nd+r}>0)\ \xrightarrow{n\to\infty}\ \rho $$
for each $r=0,\ldots,d-1$ which again proves Assertion (d).\qed
\end{proof}

\begin{Rem}\label{rem2:extra assumption}\rm
Let us finally comment on the need for the extra condition $\Erw_{i}\tau(i)^{2}<\infty$ which we have used in the estimation of $R_{n}(\delta)$ for the conclusion that
$$ n\,\max_{\vth^{-1}(\delta n)\le k\le\vth(n/\delta)}\Prob_{i}\left(\frac{\sigma(k)}{\tau_{\sigma(k)}(i)}<\pi_{i}-\eps,\,\tau_{\sigma(k)}(i)=n\right)\ =\ o(1) $$
as $n\to\infty$. An approach more in line with Doney's argument in the i.i.d.-case would be to derive this from a local limit theorem for the pair $(\sigma(k)-\pi_{i}\tau_{\sigma(k)},\tau_{\sigma(k)})$. However, this would require some knowledge of the dependence structure between $\sigma(k)$ and $\tau_{\sigma(k)}$ so as to provide the right normalization of $\sigma(k)-\pi_{i}\tau_{\sigma(k)}$. We doubt that this is possible without any extra condition on the given MRW $(M_{n},S_{n})_{n\ge 0}$.
\end{Rem}

\subsection{The case $\rho\in\{0,1\}$}

It clearly suffices to consider the case $\rho=1$ for which we make use of the following result very similar to Lemma 1 by Bertoin and Doney \cite{BertoinDoney:97} which actually goes back to Kesten as noted by them.

\begin{Lemma}
Suppose that, for any fixed $i\in\cS$, $\rho_{n}:=\Prob_{i}(S_{\tau_{n}(i)}>0)\to 1$ as $n\to\infty$. Then
\begin{equation}\label{eq:BertoinDoney Lemma 1(ii)}
\lim_{n\to\infty}\Prob_{i}(S_{k}>0\text{ for }\tau_{2n}(i)\le k\le \tau_{rn}(i))\ =\ 1
\end{equation}
for any fixed integer $r>2$.
\end{Lemma}

\begin{proof}
We adapt the argument given by Bertoin and Doney \cite[Lemma 1]{BertoinDoney:97} and prove that
\begin{equation}\label{eq:BertoinDoney Lemma 1(i)}
\Prob_{i}(S_{k}>0\text{ for }\tau_{2n}(i)\le k\le \tau_{rn}(i))\ \ge\ (1-\delta)(1-\eps)^{r}\rho_{n}^{r+1}.
\end{equation}
for any fixed integer $r>2$ and $\eps\in (0,1)$, where $\delta=\delta(\eps,n):=(1-\rho_{n})/\eps\rho_{n}\ge 0$. Obviously, this implies \eqref{eq:BertoinDoney Lemma 1(ii)}.

\vspace{.1cm}
Fix any $\eps\in (0,1)$, put $D_{n}:=\min_{0\le k\le\tau_{n}(i)}S_{k}$ for $n\in\N$ and let $q_{n}$ be the conditional $\eps$-quantile of $D_{n}$ given $S_{\tau_{n}(i)}>0$, thus $q_{n}\le 0$ and
$$ \Prob_{i}(D_{n}<q_{n}|S_{\tau_{n}(i)}>0)\ <\ \eps\ \le\ \Prob_{i}(D_{n}\le q_{n}|S_{\tau_{n}(i)}>0). $$
As a consequence,
\begin{equation}\label{eq:lower bound D_n >= q_n}
\Prob_{i}(D_{n}\le q_{n})\ \ge\ \eps\rho_{n}.
\end{equation}
Now put $\nu:=\inf\{k:D_{k}\le q_{n}\}$. Then
\begin{align*}
\Prob_{i}(S_{\tau_{n}(i)}\le 0)\ &\ge\ \Prob_{i}(S_{\tau_{n}(i)}\le 0,D_{n}\le q_{n})\\
&\ge\ \sum_{k\le n}\Prob_{i}(\nu=k)\,\Prob_{i}(S_{\tau_{n-k}(i)}\le-q_{n})\\
&\ge\ \Prob_{i}(D_{n}\le q_{n})\,\min_{0\le k\le n}\Prob_{i}(S_{\tau_{k}(i)}\le-q_{n})
\end{align*}
which in combination with \eqref{eq:lower bound D_n >= q_n} gives
$$ \min_{0\le k\le n}\Prob_{i}(S_{\tau_{k}(i)}\le-q_{n})\ \le\ \frac{1-\rho_{n}}{\eps\rho_{n}}\ =\ \delta $$
and thus $\Prob_{i}(S_{\tau_{m}(i)}>-q_{n})\ge 1-\delta$ for some integer $m=m(\eps,n)\le n$.

\vspace{.1cm}
Finally, consider the event
\begin{align*}
&S_{\tau_{n}(i)}>0,\ S_{\tau_{n+m}(i)}-S_{\tau_{n}(i)}>-q_{n},\ S_{\tau_{(s+1)n+m}(i)}-S_{\tau_{sn+m}(i)}>0,\\
&\quad\min_{0\le j\le n}\big(S_{\tau_{sn+m}(i)+j}-S_{\tau_{sn+m}(i)}\big)\ \ge\ q_{n},\quad s=1,\ldots,r
\end{align*}
on which we have $S_{k}>0$ for all $\tau_{n+m}(i)\le k\le \tau_{rn+m}(i)$. Since $m\le n$, the asserted inequality \eqref{eq:BertoinDoney Lemma 1(i)} follows.\qed
\end{proof}

In order to prove Assertion (d) of Theorem \ref{thm:MRW arcsine} given that (a)--(c) hold, choose $m=m_{n}:=\lceil(4\pi_{i})^{-1}n\rceil$ for $n\in\N$ and note that $n^{-1}\tau_{n}(i)\to\pi_{i}^{-1}$ $\Prob_{i}$-a.s. implies
$$ \lim_{n\to\infty}\Prob_{i}(\tau_{2m}(i)>n\text{ or }\tau_{8m}(i)<n)\ =\ 0. $$
Since, furthermore,
\begin{align*}
\Prob_{i}(S_{n}>0)\ &\ge\ \Prob_{i}(S_{k}>0\text{ for }\tau_{2m}(i)\le k\le\tau_{8m}(i))\ -\ \Prob_{i}(\tau_{2m}(i)>n\text{ or }\tau_{8m}(i)<n),
\end{align*}
we finally infer with the help of \eqref{eq:BertoinDoney Lemma 1(ii)}
$$ \liminf_{n\to\infty}\Prob_{i}(S_{n}>0)\ =\ 1. $$

\acknowledgement{We are most grateful to an anonymous referee for pointing out an error in an earlier version of this article. Both authors were partially supported by the Deutsche Forschungsgemeinschaft (SFB 878).}

\bibliographystyle{abbrv}
\bibliography{StoPro}

\end{document}